\documentclass[12pt]{amsart}
\usepackage{amsmath,amssymb,hyperref}

\newcommand{\bC}{\mathbb{C}}
\newcommand{\bF}{{\mathbb F}}

\newcommand{\bP}{\mathbb{P}}

\newcommand{\bR}{\mathbb{R}}
\newcommand{\bS}{\mathbb{S}}
\newcommand{\bZ}{\mathbb{Z}}

\newcommand{\cE}{\mathcal{E}}
\newcommand{\cF}{\mathcal{F}}
\newcommand{\cI}{\mathcal{I}}

\newcommand{\cN}{{\mathcal N}}
\newcommand{\cO}{\mathcal{O}}
\newcommand{\cQ}{\mathcal{Q}}
\newcommand{\cT}{{\mathcal T}}
\newcommand{\cX}{\mathcal{X}}

\newcommand{\ra}{\rightarrow}

\theoremstyle{plain}
\newtheorem{prop}{Proposition}[section]
\newtheorem{theo}[prop]{Theorem}
\newtheorem{coro}[prop]{Corollary}

\theoremstyle{remark}
\newtheorem{rema}[prop]{Remark}
\newtheorem{ques}[prop]{Question}

\theoremstyle{definition}

\newtheorem{exam}[prop]{Example}
\numberwithin{equation}{section}

\newcommand{\Bl}{\operatorname{Bl}}
\newcommand{\OGr}{\operatorname{OGr}}

\DeclareMathOperator{\Br}{Br}

\newcommand{\diag}{\operatorname{diag}}

\DeclareMathOperator{\Pic}{Pic}
\DeclareMathOperator{\CH}{CH}

\DeclareMathOperator{\Sym}{Sym}

\author{Brendan Hassett}
\address{Department of Mathematics\\
Brown University \\
Box 1917 
151 Thayer Street
Providence, RI 02912 \\
USA}
\email{brendan\underline{ }hassett@brown.edu}

\author{J\'anos Koll\'ar}
\address{Department of Mathematics\\
Princeton University\\
Fine Hall, Washington Road \\
Princeton NJ 08544-1000\\
USA}
\email{kollar@math.princeton.edu}

\author{Yuri Tschinkel}
\address{Courant Institute\\
                New York University \\
                New York, NY 10012 \\
                USA }
\address{Simons Foundation\\
160 Fifth Avenue\\
New York, NY 10010\\
USA}

\email{tschinkel@cims.nyu.edu}

\title[Real intersections of quadrics]{Rationality of  even-dimensional intersections of two real quadrics}

\begin{document}

\begin{abstract}
We study rationality constructions for smooth complete intersections of two quadrics over nonclosed fields.
Over the real numbers, we establish a criterion for rationality in dimension four. 
\end{abstract}

\maketitle

\section{Introduction}
\label{sect:intro}

Consider a geometrically rational variety $X$, smooth and
projective over a field $k$. 
Is $X$ rational over $k$? A necessary condition is that $X(k)\neq \emptyset$,
which is sufficient in dimension one, as well as for quadric hypersurfaces
and Brauer-Severi varieties of arbitrary dimension. 
When the dimension of $X$ is at most two, rationality over $k$ 
was settled by work of Enriques, Manin, and Iskovksikh \cite{Isk}.
Rationality is encoded in the Galois action on the geometric 
N\'eron-Severi group -- varieties with rational points that are
`minimal' in the sense of birational geometry need not be 
rational. In dimension three, 
recent work \cite{HT1,HT2,KP1,BenWit1,BenWit2}
has clarified the criteria for rationality: one also needs to take
into account principal homogeneous spaces over the intermediate 
Jacobian, reflecting which curve classes are realized over the ground
field. 
The case of complete intersections of two quadrics was an
important first step in understanding the overall structure
\cite{HTplusCT}; rationality in dimension three
is equivalent to the existence of a line over $k$ \cite{HT1,BenWit2,KP1}.

These developments stimulate investigations in higher dimensions \cite{KP2};
the examples considered are rational provided there are rational points.
In this paper, we focus on the case of four-dimensional complete
intersections of two quadrics, especially over the real numbers $\bR$.
Here we exhibit rational examples without lines and explore further
rationality constructions. 

\begin{theo}
\label{theo:main}
A smooth complete intersection of two quadrics $X\subset \bP^6$ over $\bR$ is
rational if and only if $X(\bR)$ is nonempty and connected.  
\end{theo}

In general, a projective variety $X$ that is rational over $\bR$ has connected 
nonempty real locus $X(\bR)$. The point of Theorem~\ref{theo:main} is
that this necessary condition is also sufficient.

\begin{coro}
A smooth complete intersection of two quadrics
$X\subset \bP^6$ is rational over $\bR$ if and only if
there exists a unirational parametrization
$\bP^4 \dashrightarrow X$, defined 
over $\bR$, of odd degree.
\end{coro}
Indeed, odd degree rational maps are
surjective on real points, which guarantees that $X(\bR)$ is 
connected. Smooth complete intersections of two quadrics, of dimension at least two, are
unirational provided they have a rational point; see Section~\ref{subsect:genrat}
for references and discussion.

We also characterize rationality in dimension six, with
the exception of one isotopy class that remains open 
(see Section~\ref{subsect:135}).

\

Here is the roadmap of this paper. In Section~\ref{sect:back} we recall basic facts about quadrics in even-dimensional projective spaces and their intersections. All interesting cohomology is spanned by the classes of projective subspaces in $X$ of half-dimension, and the Galois group
acts on these classes via symmetries of the primitive part of this cohomology, a lattice for the root system $D_{2n+3}$. 
In Section~\ref{sect:rat-con} we present several rationality constructions. 
The isotopy classification, using Krasnov's invariants, is presented in Section~\ref{sect:isotopy}; we draw
connections with the Weyl group actions. 
In Section~\ref{sect:rev} we focus attention on cases where rationality is
not obvious, e.g., due to the presence of a line.
In Section~\ref{sect:application} we prove Theorem~\ref{theo:main} and discuss the applicability and limitations 
of our constructions in dimensions four and six.
We speculate on possible extensions to more general fields in Section~\ref{sect:extensions}.

\medskip
\noindent
\textbf{Acknowledgments:}
The first author was partially supported by NSF grant 1701659 and Simons Foundation Award 546235, 
the second author by the NSF grant DMS-1901855, and the 
third author by NSF grant 2000099.

\section{Geometric background}
\label{sect:back}

\subsection{Roots and weights}

Let $D_{2n+3}$ be the root lattice of the corresponding 
Dynkin diagram, expressed in the standard Euclidean lattice
$$\left<L_1,\ldots,L_{2n+3}\right>, \quad \quad  L_i\cdot L_j = \delta_{ij}$$
as the lattice generated by simple roots
\begin{align*}
R_1=L_1-L_2,R_2=L_2-L_3, &\ldots,R_{2n+1}=L_{2n+1}-L_{2n+2}, \\
R_{2n+2}=L_{2n+2}-L_{2n+3}, \quad &R_{2n+3}=L_{2n+2}+L_{2n+3}.
\end{align*}
Its discriminant group is cyclic of order four, generated by
$$\tfrac{1}{4}
(2(R_1+2R_2+\cdots+(2n+1)R_{2n+1})+(2n+1)R_{2n+2}+(2n+3)R_{2n+3}).$$
Multiplication by $-1$ acts on the discriminant via $\pm 1$.
The outer automorphisms of $D_{2n+3}$ also act via automorphisms
of $D_{2n+3}$ acting on the discriminant via $\pm 1$, e.g.,
exchanging $R_{2n+2}$ and $R_{2n+3}$ and keeping the other roots
fixed.

The Weyl group $W(D_{2n+3})$ acts in the
basis $\{L_i\}$ via signed permutations with an even number of $-1$ entries.
The outer automorphisms act via signed permutations with no
constraints on the choice of signs, e.g.,
$$L_i \mapsto L_i,\quad  i=1,\ldots,2n+2, \quad \quad L_{2n+3} \mapsto -L_{2n+3}.$$

The odd and even half-spin representations have weights indexed by
subsets $I \subset \{1,2,\ldots,2n+3 \}$, with $|I|$ odd or even,
written
$$w_I= \tfrac{1}{2} (\sum_{i\in I} L_i - \sum_{j \in I^c}L_j).$$
The odd and even weights are exchanged by outer automorphisms.

\subsection{Planes}
\label{subsect:planes}
In this section, we assume that the ground field is algebraically closed of characteristic zero. 

Let $X\subset \bP^{2n+2}$ be a smooth complete intersection of two quadrics.
We will identify subvarieties in $X$ with their classes in the cohomology of $X$ when no confusion may arise.

Let $h$ denote the hyperplane section and consider the
primitive cohomology of $X$ under the intersection pairing.
Reid \cite[3.14]{ReidThesis} shows that
$$(h^n)^{\perp}\simeq (-1)^n D_{2n+3}.$$
In other words, the primitive sublattice of $H^{2n}(X,\bZ(n))$
-- the Tate twist of singular cohomology for the underlying
complex variety -- may be
identified with the root lattice.   
This is the target of the cycle class map 
$$\CH^n(X) \ra H^{2n}(X,\bZ(n))$$
so the sign convention is natural.

\begin{rema}[Caveat on signs]
When $X$ is defined over $\bR$, codimension-$n$ 
subvarieties $Z\subset X$ defined over
$\bR$ yield classes in $H^{2n}(X,\bZ(n))$ that are invariant
under complex conjugation. However, the corresponding classes
in $H^{2n}(X,\bZ)$ are multiplied by $(-1)^n$. When we mention
invariant classes, it is with respect to the former 
action. 
\end{rema}

Given a plane $P\simeq \bP^n \subset X$, we have 
$$
(P\cdot P)_X=c_n
$$
where  \cite[3.11]{ReidThesis}
$$c_0=1, c_1=-1, c_2=2, c_3=-2, \ldots,
c_n=(-1)^n(\lfloor \tfrac{n}{2}\rfloor + 1).$$
The projection of $P$ into rational primitive cohomology takes
the form
$$P-\tfrac{1}{4}h^n$$
which has self-intersection
$c_n-1/4$. The corresponding element $w_P \in D_{2n+3}$
has 
$$w_P\cdot w_P = 
(2n+3)/4.$$
By \cite[Cor.~3.9]{ReidThesis}, we obtain bijections 
$$\{w_P\}_{P\simeq \bP^n \subset X } 
= \{w_I\}_{|I| \text{ has fixed parity }}.$$
Note that the residual intersections to $\bP^n \subset X$
$$X\cap \bP^{n+2} = \bP^n \cup S$$
give cubic scrolls $S\subset X$; these realize the weights of 
opposite parity.

By \cite[Th.~3.8]{ReidThesis}, two planes $P_1$ and $P_2$ are disjoint if and only if
$$
w_{P_1}\cdot w_{P_2}=(-1)^{n+1}/4.
$$ 
For example, if $n=1$ and 
$w_{P_1}$ is identified with $(L_1-L_2-L_3-L_4-L_5)/2$
then the relevant weights are 
$$(L_1+L_2+L_3-L_4-L_5)/2,  \quad \ldots ,  \quad (L_1-L_2-L_3+L_4+L_5)/2$$
and 
$$(-L_1+L_2+L_3+L_4-L_5)/2, \quad \ldots ,  \quad (-L_1-L_2+L_3+L_4+L_5)/2,$$
a total of $10=\binom{5}{2}$ such lines.
When $n=2$ and $w_{P_1}$ is identified with $(L_1-L_2-\cdots-L_7)/2$
then the relevant weights are
$$(L_1+L_2+L_3+L_4+L_5-L_6-L_7)/2,\ldots,(L_1-L_2-L_3+L_4+L_5+L_6+L_7)/2$$
and
$$(-L_1+L_2+L_3+L_4-L_5-L_6-L_7)/2,\ldots,(-L_1-L_2-L_3-L_4+L_5+L_6+L_7)/2,$$
a total of $\binom{6}{4}+\binom{6}{3}=35=\binom{7}{3}$ planes.
The planes $P_1$ and $P_2$ meet at a point if and only if
$$w_{P_1}\cdot w_{P_2}=(-1)^n \tfrac{3}{4}.$$
If they meet along an $r$-plane then an excess intersection
computation gives \cite[3.10]{ReidThesis}
$$P_1\cdot P_2 = (-1)^r (\lfloor \tfrac{r}{2} \rfloor)+1)$$
and 
$$w_{P_1}\cdot w_{P_2}=(-1)^{r+n}(\lfloor \tfrac{r}{2} \rfloor + 1) - (-1)^n\tfrac{1}{4}.$$

In particular, they meet along an $(n-1)$-plane when
$$w_{P_1}\cdot w_{P_2}=-(\lfloor \tfrac{n-1}{2} \rfloor +1 ) - (-1)^n\tfrac{1}{4};$$
for a fixed $w_{P_1}$ there are $2n+3$ planes $P_2$ meeting $P_1$ in this way.
For example, if $n=1$ and 
$w_{P_1}=(L_1-L_2-L_3-L_4-L_5)/2$ then the possibilities for $w_{P_2}$ are
\begin{align*}
(L_1+L_2+L_3+L_4+L_5)/2,  \quad (-L_1-L_2+L_3+L_4+L_5)/2,\\
 (-L_1+L_2-L_3+L_4+L_5)/2,  \quad
(-L_1+L_2+L_3-L_4+L_5)/2, \\
(-L_1+L_2+L_3+L_4-L_5)/2.
\end{align*}

\subsection{Quadrics}
\label{subsect:quadrics}
We retain the notation of Section~\ref{subsect:planes}.

Our next task is to analyze quadric $n$-folds $Q\subset X$, i.e.,
$Q$ a degree-two hypersurface in $\bP^{n+1}$.
Let $\{\cQ_t \}, t\in \bP^1$ denote the pencil of quadric hypersurfaces
cutting out $X$. The degeneracy locus
$$D:= \{t \in \bP^1: \cQ_t \text{ singular } \}$$
consists of $2n+3$ points; since $X$ is smooth, each has rank $2n+2$. 
The Hilbert scheme of quadric $n$-folds $Q\subset X$ is isomorphic to
the relative Fano variety of $(n+1)$-planes
$$\cF(\cQ/\bP^1) = \{ \Pi\simeq \bP^{n+1} \subset \cQ_t \text{ for some }
t \in \bP^1 \},$$
which consists of $2(2n+3)$ copies of the connected isotropic Grassmannian
$\OGr(n+1,2n+2)$. Given a quadric $Q$, its projection to rational 
primitive cohomology
$$Q - \tfrac{1}{2}h^n$$
corresponds to an element
$$w_Q \in D_{2n+3}, \quad w_Q \in \{\pm L_1,\ldots,\pm L_{2n+3}\}.$$
In particular, we have
$$Q\cdot Q = 
\begin{cases} 2 &      \text{ if $n$ even} \\
		    0 & \text{ if $n$ odd.}
		\end{cases}
$$
Residuation in a complete intersection of linear forms
$$Q \cup Q' = X \cap h^n$$
reverses signs, i.e., $w_Q = - w_{Q'}$.
On the other hand, if $Q_1$ and $Q_2$ are not residual then
\begin{equation} \label{notresidual}
Q_1\cdot Q_2 = 1.
\end{equation}

We summarize this as follows:
\begin{prop} \label{prop:permquad}
The signed permutation representation of $W(D_{2n+3})$ is realized
via the action on classes $[Q]$, where $Q\subset X$ is a 
quadric $n$-fold.
\end{prop}

Note that there are 
$$2^{2n+1}(2n+3)$$
reducible quadrics -- unions of two $n$-planes meeting in an
$(n-1)$-plane -- with $2^{2n}$ reducible quadrics
in each copy of the isotropic Grassmannian.

\section{Rationality constructions}
\label{sect:rat-con}

We now work over an arbitrary field $k$ of characteristic zero.
\subsection{General considerations}
\label{subsect:genrat}
Let $X\subset \bP^{d+2}$ denote a smooth complete intersection of two
quadrics of dimension at least two. Recall the following:
\begin{itemize}
\item If $X(k)\neq \emptyset$ then
$X$ is unirational over $k$ and has Zariski dense rational points
\cite[Rem.~3.28.3]{CTSSDI}.
\item If there is a line $\ell \subset X$ defined over $k$ then
projection induces a birational map
$\pi_{\ell}: X \stackrel{\sim}{\dashrightarrow} \bP^d$.
\end{itemize}

For reference, we recall Amer's theorem \cite[Th.~2.2]{Leep}:
\begin{theo} \label{theo:AB}
Let $k$ be a field of characteristic not two, $F_1$ and $F_2$
quadrics over $k$, and $\cQ_t = \{F_1 + t F_2\}$ the associated
pencil of quadrics over $k(t)$. Then $X=\{F_1=F_2=0\}$ has an
$r$-dimensional isotropic subspace over $k$ if and only if
$\cQ_t$ has an $r$-dimensional isotropic subspace over $k(t)$.
\end{theo}

We apply this for $k=\bR$, where $X\subset \bP^{d+2}$ is a smooth complete
intersection of two quadrics and $\cQ \ra \bP^1$ is the associated
pencil.

Recall Springer's Theorem: A quadric hypersurface $\cQ$ over a field
$L$ has a rational point if it admits a rational point over some
odd-degree extension of $L$. Applying this to the pencil
$\cQ \ra \bP^1$ associated with a complete intersection of two 
quadrics, with Amer's Theorem, yields:
\begin{prop}
If $d\ge 1$ and $X$ contains a 
subvariety of 
odd degree over $k$ then $X(k) \neq \emptyset$.
\end{prop}
We can prove a bit more:
\begin{prop}
If $d\ge 2$ and 
$X$ has a curve of odd degree defined over $k$ then $X$
is rational over $k$.
\end{prop}
\begin{proof}
Recall that double projection from a sufficiently general rational
point $x\in X(k)$ yields a diagram
$$ X \stackrel{\sim}{\dashrightarrow} Y \ra \bP^1$$
where $Y$ is a quadric bundle of relative dimension $d-1$. 
A curve $C\subset X$ of odd degree yields an multisection of
this bundle of odd degree. Thus $Y\ra \bP^1$ has a section
by Springer's Theorem and its generic fiber $Y_t$ is rational over $k(\bP^1)$. 
It follows that $X$ is rational over the ground field.
\end{proof}

\begin{rema} \label{rema:Witt}
The pencil defining $X$ gives a quadric bundle
$$\cQ \ra \bP^1$$
of relative dimension $d+1$. We apply Witt's decomposition theorem
to $[\cQ_t]$ and $[Y_t]$, understood as quadratic forms over 
$k(\bP^1)=k(t)$, to obtain
$$[\cQ_t] = [Y_t] \oplus \left( \begin{matrix} 0 & 1 \\ 1 & 0 \end{matrix}
\right).$$
Thus a section of $Y\ra \bP^1$ yields an isotropic line of $\cQ \ra \bP^1$,
and Theorem~\ref{theo:AB} implies that $X$ contains a line defined over $k$.
\end{rema}

\begin{coro} (see appendix by Colliot-Th\'el\`ene \cite[Th.~A5]{HTplusCT})
Let $X\subset \bP^{d+2}$ denote a smooth complete intersection of
two quadrics of dimension at least two. Suppose there exists an
irreducible positive-dimensional subvariety
$W\subset X$ of odd degree, defined over $k$.
Then $X$ is rational over $k$.
\end{coro}

Given these results, we focus on proving rationality in cases where $X$ 
does not contain lines or positive-dimensional subvarieties of odd degree.

\subsection{Rationality using half-dimensional subvarieties}
\label{sect:half}

We now turn to even-dimensional intersections of two quadrics
$$
X\subset \bP^{2n+2}, \quad n\ge 1. 
$$
Throughout, we assume that 
$X(k)\neq \emptyset,$  and thus $X$ is $k$-unirational and $k$-rational points are Zariski dense.

\

\noindent
{\bf Construction I:}
Suppose that 
\begin{itemize}
\item{$X$ has a pair of conjugate disjoint $n$-planes $P,\bar{P}$,
defined over a quadratic extension $K$ of $k$.}
\end{itemize}
Projecting from a general $x\in X(k)$ gives a birational map 
$$X \stackrel{\sim}{\dashrightarrow} X' \subset \bP^{2n+1},$$
where $X'$ is a cubic hypersurface. 

Since $X(k) \subset X$ is Zariski dense, we may assume that
the images of $P$ and $\bar{P}$ in $X'$ remain disjoint. The
`third point' construction gives a birational map
$$\mathbf{R}_{K/k}(P) \stackrel{\sim}{\dashrightarrow} X',$$
where the source variety is the restriction of scalars. 
We conclude that $X$ is rational over $k$. 
This construction appears in \cite[Th.~2.4]{CTSSDI}.

\

\noindent
{\bf Construction II:}
Suppose that
\begin{itemize}
\item{$X$ has a pair of conjugate disjoint quadric $n$-folds $Q,\bar{Q}$,
defined over a quadratic extension $K$ of $k$, and meeting transversally
at one point.}
\end{itemize}
Projecting from $x\in Q\cap \bar{Q}$, which is a $k$-rational
point $X$, gives a birational map 
$$X \stackrel{\sim}{\dashrightarrow} X' \subset \bP^{2n+1},$$
where $X'$ is a cubic hypersurface. 

The proper transforms $Q',\bar{Q}'\subset X'$ are disjoint unless there
exists a line 
$$x\in \ell \subset \bP^{2n+2}$$
with
$$\{x\} \subsetneq \ell \cap Q , \ell \cap \bar{Q}$$
as schemes.  
We may assume that $\ell \not \subset X$ as we already know $X$ is rational
in this case. Thus the only possibility is 
$$\ell \cap Q = \ell \cap \bar{Q}$$
as length-two subschemes, which is precluded by the intersection assumption.

Repeating the argument for Construction I thus gives rationality.

\

\noindent
{\bf Construction III:}
Suppose that 
\begin{itemize}
\item 
$X$ contains a quadric $Q$ of dimension $n$, defined over $k$.
\end{itemize}
Projection gives a fibration
$$
        q:\Bl_Q(X) \rightarrow \bP^n
$$
with fibers quadrics of dimension $n$. Now suppose that $X$ 
contains a second $n$-fold
$T$ with the property that
$$
\deg(T) - T\cdot Q
$$
is odd, i.e., a multisection for $q$ of odd degree. It follows that the 
generic fiber of $q$ is rational and thus $Y$ is rational over $k$.

When $\dim(X)=4$
a number of $T$ might work, e.g., 
\begin{itemize}
\item a plane disjoint from $Q$,
\item a second quadric meeting $Q$
in one point, 
\item a quartic or a sextic del Pezzo surface meeting $Q$ in one or three points, 
\item a degree 8 K3 surface meeting $Q$ in one, three, five, or seven points.
\end{itemize}

\

\noindent
{\bf Construction IV:}
Suppose that 
\begin{itemize}
\item $\dim(X)=4$ and $X$
contains a quartic scroll $T$, defined over $k$. 
\end{itemize}
Geometrically, 
$T$ is the image of
the ruled surface
$$\bF_0=\bP^1 \times \bP^1 \hookrightarrow \bP^5$$
under the linear series of bidegree $(1,2)$.
Over $\bR$
we are interested in cases where $T=\bP^1 \times C$ with $C$ a 
nonsplit conic. We do not want to force $X$ to 
have lines! 
(Note that quartic scrolls geometrically isomorphic to $\bF_2$
contain lines defined over the ground field and thus are not
useful for our purposes.)

On projecting from a point $x\in X(k)$ we get a cubic
fourfold 
$$
X'\subset \bP^5,
$$ 
containing a quartic scroll.
The Beauville-Donagi construction \cite{BD} -- concretely,
take the image under the linear system of quadrics
vanishing along $T$ --
shows that $X'$ is birational to a quadric hypersurface
thus rational over $k$.

\

Recall that an $n$-dimensional smooth variety $W \subset \bP^{2n+1}$ is said to have
{\em one apparent double point} if a generic point is contained
in a unique secant to $W$.

\

\noindent
{\bf  Construction V:}
Suppose that $X$ contains a variety $W$ defined over $k$ of dimension
$n\ge 2$ such that:
\begin{itemize}
\item $W$ spans a $\bP^{2n+1}$ and has one apparent double point; or
\item $W$ has a rational point $w$ such that projecting from $w$ maps $W$ birationally to a variety with one apparent double point.
\end{itemize}
Then $X$ is rational over $k$. 

As before, one projects from a rational point to a get cubic 
hypersurface $X' \subset \bP^{2n+1}$. 
Cubic hypersurfaces containing varieties $W$ with
one apparent double point are rational \cite[Prop.~9]{Russo}.
Indeed, intersecting secant lines of $W$ with $X'$ yields
$$\Sym^2(W) \dashrightarrow X',$$
which is birational when each point lies on a unique secant to $W$.

Quartic scrolls in $\bP^5$ have one apparent
double point so Construction IV is a special case of Construction V.

\section{Isotopy classification}
\label{sect:isotopy}
We review the classification of smooth complete intersections of two
quadrics $X\subset \bP^{2n+2}$ over $\bR$, following \cite{Krasnov18}.

Express $X=\{F_1=F_2=0\}$ where $F_1$ and $F_2$ are real quadratic forms.
We continue to use $D$ for the degeneracy locus of the associated pencil
$\cQ_t = \{t_1F_1+t_2F_2=0\}$.
Let $r=|D(\bR)|$ which is odd with $r \le 2n+3$. 
Consider the signatures $(I^+,I^-)$ of the forms
$$s_1F_1+s_2F_2, \quad (s_1,s_2) \in \bS^1 = \{(s_1,s_2) \in \bR^2:s_1^2+s_2^2=1\}.$$
Record these at the $2r$ points lying over $D$, in order as we trace the circle counterclockwise.
We label each of these points with $\pm$ depending on whether the positive part $I^+$ of the
signature increases or decreases as we cross the point. Each point of $D(\bR)$ yields a pair 
of antipodal points on $\bS^1$ labelled with opposite signs.
For example, for $n=0$ and $r=3$ admissible sequences of signatures and $\pm$'s include
$$
(0,2) (1,1) (2,0) (2,0) (1,1) (0,2) \quad (+,+,+,-,-,-).$$
and 
$$(1,1) (1,1) (1,1) (1,1) (1,1) (1,1) \quad (+,-,+,-,+,-).$$

Suppose the sequence of $\pm$'s has maximal unbroken chains of $+$'s of lengths
$r_1, r_2, \ldots, r_{2s+1}$ where
$$r=r_1+r_2+\cdots+r_{2s+1}.$$
The number of terms is odd because antipodal points have opposite signs.
In the examples above we have $3=3$ and $3=1+1+1$.  
Our invariant is the sequence $(r_1,\ldots,r_{2s+1})$ up to cyclic permutations and
reversals -- a complete isotopy invariant of $X$ over $\bR$ \cite{Krasnov18}.

We derive a sequence of $\pm 1$'s of length $k$ from this invariant as follows: For each point of $D(\bR)$,
record the sign of the discriminant of the associated rank-$(2n+2)$ quadratic form, determined by the 
parity of $(I^+-I^-)/2$.   
In the examples above, we obtain $(-1,+1,-1)$ and $(+1,+1,+1)$. The number of $-1$'s is always even.

The analysis in Section~\ref{subsect:quadrics} shows that complex conjugation acts on $H^{2n}(X,\bZ(n))$
in the basis $\{L_1,\ldots,L_{2n+3}\}$ as a signed permutation of order two. This is a direct sum of blocks
$$(+1), (-1), \pm \left( \begin{matrix} 0 & 1 \\
							1 & 0 \end{matrix} \right).$$
Actually, we may assume the sign is positive in the third case after conjugating by 
$$\left( \begin{matrix} 0  & -1 \\
			          1 & 0 \end{matrix} \right).$$
Suppose there are $a$ blocks $(+1)$, $2b$ blocks $(-1)$, and $c$ blocks of the third kind, with
$a+2b+2c=2n+3$.
These correspond to the conjugacy classes of involutions $\iota \in W(D_{2n+3})$ \cite[\S 3.2,3.3]{Kottwitz}.
We have $r=a+2b$, reflecting the number of points of $D(\bR)$ with positive and negative discriminants 
respectively, and $2c=2n+3-r$, reflecting the number of complex-conjugate pairs in $D(\bC)\setminus D(\bR)$.

The passage from isotopy classes to conjugacy classes in $W(D_{2n+3})$ results in a loss of information.
We give an example for real quartic del Pezzo surfaces $X\subset \bP^4$.

\begin{exam} \label{exam:surface}
The isotopy class $(5)$ has singular members with signatures 
$$(0,4) (1,3) (2,2) (3,1) (4,0) (4,0) (3,1) (2,2) (1,3) (0,4)$$
with involution given the diagonal $5\times 5$ matrix
$$\diag(1,-1,{\bf{1}},-1,1),$$
where the bolded $\bf{1}$ corresponds to a degenerate fiber $\cQ_t$ whose rulings sweep out quadric curves
(conics) on $X$ defined over $\bR$. (There is only one pair of such conics.) Here $X(\bR)=\emptyset$ as it is contained in an anisotropic
quadric threefold.

The isotopy class $(2,2,1)$ has singular members with signatures
$$(1,3) (2,2) (2,2) (2,2) (3,1) (3,1) (2,2) (2,2) (2,2) (1,3)$$
with involution
$$\diag(-1,{\bf{1}},{\bf{1}},{\bf{1}},-1).$$
This has the same Galois action but contains three pairs of conics defined over $\bR$, represented by the bolded
$\bf{1}$'s.

These are distinguished cohomologically by the arrow
$$\bZ^3 \simeq H^0(G, \Pic(X_{\bC})) \ra H^2(G, \Gamma(\cO^*_{X_{\bC}}))=\Br(\bR)\simeq \bZ/2\bZ$$
of the Hochschild-Serre spectral sequence.
\end{exam}

\begin{prop}
Fix a conjugacy class $[\iota=\iota_{abc}]$ of involutions in $W(D_{2n+3})$. Consider the isotopy classes of 
$X(\bR)\subset \bP^{2n+2}$ such that complex conjugation acts by $\iota$. The possible isotopy classes correspond to shuffles of
$$(\underbrace{1,\ldots,1}_{a \text{ times}},\underbrace{-1,\ldots,-1}_{2b \text{ times}})$$
up to cyclic permutations and reversals.
\end{prop}
\begin{proof}
Observe first that points in $D(\bC)\setminus D(\bR)$ are irrelevant to the Krasnov invariant, 
assuming the dimension is given. So it makes sense they are not relevant in the enumeration.

We have already seen that sequences of the prescribed form arise from each isotopy class; we present
the reverse construction.

Choose such a sequence, e.g.
$$(1,1,1,-1,1,-1,-1,1,-1).$$
The key observation is that local maxima and minima of $I^+-I^-$ -- which necessarily occur at
smooth points -- arise precisely between points of the degeneracy locus where signs do {\em not} change.
We indicate smooth fibers achieving maxima/minima with $\mid$, e.g.,
$$ 1\mid 1 \mid 1,-1,1,-1 \mid -1,1,-1$$
or equivalently
$$\mid 1 \mid 1,-1,1,-1 \mid -1,1,-1,1\mid $$
after cyclic permutation.

Lifting to the double cover $\bS^1$ entails concatenating two such expressions:
$$\mid 1 \mid 1,-1,1,-1 \mid -1,1,-1,1
\mid 1 \mid 1,-1,1,-1 \mid -1,1,-1,1\mid$$
From this, we read off the points of $D(\bR)$ on which $I^+$ increases and decreases
$$+----++++-++++----$$
which determines the Krasnov invariant -- $(1,4,4)$ in this example.
\end{proof}

\section{Applying quadratic forms}
\label{sect:rev}

\subsection{Quadric fibrations over real curves}
\label{sect:Witt}

Let $C$ be a smooth projective geometrically connected curve
over $\bR$ with function field $K=\bR(C)$. Let $Q \subset \bP^{d+1}$
be a smooth (rank $d+2$) quadric hypersurface over $K$ and $F_i(Q)$
the $i$-dimensional isotropic subspaces, so that $F_0(Q)=Q$
and $F_m(Q)$ is empty when $2m > d$. If $d=2m$ then
$F_m(Q)$ has two geometrically connected components; otherwise it is
connected.  

Suppose that $\pi:\cQ \ra C$ is a regular projective model of $Q$,
such that the fibers are all quadric hypersurfaces of rank at least $d+1$.  
The locus $D \subset C$ corresponding to fibers of rank $d+1$ is
called the {\em degeneracy locus}.

Fundamental results of Witt -- see \cite[\S 2]{CT} and \cite{Sch} 
for a modern formulation in terms of local-global principles
-- assert that:
\begin{itemize}
\item{if $d>0$ then $Q(K)\neq \emptyset$ if $\cQ_c=\pi^{-1}(c)$ 
has a smooth real point for each $c\in C(\bR)$;}
\item{if $d>2$ then $F_1(Q)(K) \neq \emptyset$ if $F_1(\cQ_c)$
has a smooth real point for each $c\in C(\bR)$.}
\end{itemize}
We can translate these into conditions on the signatures of the
smooth fibers
\begin{itemize}
\item{if $d>0$ then $Q(K) \neq \emptyset$ if $\cQ_c$ is not
definite for any $c\in (C\setminus D)(\bR)$;}
\item{if $d>2$ then $F_1(Q)(K) \neq \emptyset$ if $\cQ_c$ does
not have signatures $(d+2,0),(d+1,1),(1,d+1)$ or $(0,d+2)$
for any $c\in (C\setminus D)(\bR)$.}
\end{itemize}
In other words, we have points and lines over $K$ if the fibers
permit them.

This reflects a general principle: Suppose $\cX$ is regular and
has a flat proper morphism $\varpi:\cX \ra C$ to a curve $C$,
all defined over $\bR$.
The local-global and reciprocity obstructions to sections of 
$\varpi$ are reflected in the absence of {\em continuous} sections 
$C(\bR) \ra \cX(\bR)$ for induced map of the underlying
real manifolds \cite{Ducros}.

\subsection{Implications of Amer's Theorem}
\label{sect:amer}
Let $X\subset \bP^{d+2}$ be a smooth complete intersection of two quadrics over $\bR$
and $\cQ \ra \bP^1$ the associated pencil of quadrics.

The results
of Section~\ref{sect:Witt} imply that 
$\cQ \ra \bP^1$ has a section unless the Krasnov invariant
is $(d+3)$; the variety of lines $F_1(\cQ/\bP^1)\ra \bP^1$
has a section unless the Krasnov invariant is
$$(d+3), 
(d+2-e,e,1) \text{ with } 1\le e \le \tfrac{d+2}{2}, \quad 
(d+1).$$
Thus the Krasnov invariant determines 
which dimensional linear subspaces and quadrics appear on $X$:
\begin{prop} \label{prop:noline}
Let $X \subset \bP^{d+2}$ be a smooth complete intersection of two quadrics defined
over $\bR$.
The only isotopy classes of $X$ that fail to contain a line are:
\begin{itemize}
\item{$(d+3)$ where $X(\bR)=\emptyset$;}
\item{$(d+2-e,e,1)$ with $1\le e \le \tfrac{d+2}{2}$;}
\item{$(d+1)$.}
\end{itemize}
\end{prop}
The case $(d+1,1,1)$ has disconnected real 
locus $X(\bR)$ \cite[p.~117]{Krasnov18}; thus $X$ cannot be rational
over $\bR$. The cases $(d+2-e,e,1)$ with
$2\le e \le \tfrac{d+2}{2}$ are connected.

\subsection{Quadric $n$-folds over $\bR$}
\label{subsect:overR}
Assume now that $X$ has even dimension $d=2n$.
We may read off from the invariant $(r_1,\ldots,r_{2s+1})$ 
which classes of quadric $n$-folds 
$Q \subset X_{\bC}$
are realized by quadrics defined over $\bR$.

Fix a smooth real quadric hypersurface 
$${\bf Q}=\{F=0\} \subset \bP^{2n+1}.$$
The following conditions are equivalent:
\begin{itemize}
\item
the geometric components of the variety of maximal isotropic subspaces 
$\OGr(\bf{Q})$ are defined over $\bR$;
\item the discriminant of $F$ is positive;
\item the $I^+(F)-I^-(F)$ is divisible by four.
\end{itemize}
This means that complex conjugation fixes the {\em class} of a maximal
isotropic subspace. We also have equivalence among:
\begin{itemize}
\item there is a maximal isotropic subspace $\bP^n \subset \bf{Q}$ defined
over $\bR$;
\item the signature of $F$ is $(n+1,n+1)$.  
\end{itemize}

Thus quadric $n$-folds 
$$Q \subset X \subset \bP^{2n+2}$$
defined over $\bR$ correspond to rulings of degenerate fibers 
$\cQ_t, t\in D(\bR)$ where $\cQ_t$ has signature $(n+1,n+1)$.
As in Example~\ref{exam:surface}, the corresponding $(+1)$-blocks
in the complex conjugation involution $\iota \in W(D_{2n+3})$
will be designated $\bf{1}$, in boldface.

\subsection{Analysis of the remaining even-dimensional isotopy classes}
\label{subsect:remaining}

We continue to assume that $X$ has even dimension $d=2n$, focusing on the isotopy classes without lines.

The cases 
$$
(2n+2-e,e,1)=(e,1,2n+2-e),  \quad 2\le e \le n+1
$$ 
have degeneracy
consisting of $2n+3$ real points.
The signatures of nonsingular members are
\begin{align*}
& (1,2n+2) \ldots (e+1,2n+2-e) (e,2n+3-e) (e+1,2n+2-e) \ldots \\
& (2n+1,2) (2n+2,1) \ldots (2n+2-e,e+1) (2n+3-e,e) \\
& (2n+2-e,e+1) \ldots (2,2n+1)
\end{align*}

For $(n+1,n+1,1)$ the resulting
signed permutation matrix is the diagonal matrix
\begin{equation} \label{eisn+1}
\diag(\underbrace{(-1)^n,\ldots,-1,{\mathbf 1}}_{n+1 \text{ terms }},
{\mathbf 1},\underbrace{{\mathbf 1},-1,\ldots,(-1)^n}_{n+1 \text{ terms }}),
\end{equation}
with the bolded $\bf{1}$'s corresponding to singular fibers with
signature $(n+1,n+1)$.
The number of $+1$'s 
$$a=\begin{cases} n+2 & \text{if $n$ odd }\\
                n+3 &  \text{if $n$ even.}
	\end{cases}
$$
For $e\neq n+1$ we have 
$$
\diag(\underbrace{(-1)^n,\ldots,(-1)^{n+e-1}}_{e \text{ terms}},(-1)^{n+e-1},
\underbrace{(-1)^{n+(2n+2-e)-1},\ldots,(-1)^n}_{2n+2-e \text{ terms}}).
$$
Note that $(-1)^{n+e-1}=(-1)^{n+(2n+2-e)-1}$ so the three middle terms are equal.
There is exactly one $\bf{1}$ corresponding to the singular fiber with
signature $(n+1,n+1)$.
The number of $+1$'s is given
\begin{equation} \label{countsigns}
a=\begin{cases} n & \text{if $n,e$ odd }\\
		  n+2 &	\text{if $n$ odd and $e$ even} \\
		  n+3 &	\text{if $n$ even and $e$ odd} \\
                  n+1 & \text{if $n,e$ even.}
	\end{cases}
\end{equation}

For case $(2n+1)$ the
signatures of nonsingular members are
\begin{align*}
(2,2n+1) \ldots (2n,3) (2n+1,2) 
(2n+2,3)\ldots (2,2n+1)
\end{align*}
The signed permutation matrix has one factor
$$\left( \begin{matrix} 0 & 1 \\
			1 & 0 \end{matrix} \right)
$$
and diagonal entries
$$((-1)^{n-1},\ldots,-1,1,-1,\ldots,(-1)^{n-1}).$$
The number of positive terms is
$$a = \begin{cases} n & \text{if $n$ odd} \\
		    n+1 & \text{if $n$ even.}
      \end{cases}
$$

\section{Application of the constructions}
\label{sect:application}
\subsection{Proof of Theorem~\ref{theo:main}}
Rationality is evident for isotopy classes
of varieties that contain a line defined over $\bR$. 
Propostion~\ref{prop:noline} enumerates the remaining cases
$$(5),(1,3,3),(1,2,4).$$
These are covered by the following propositions. 

\begin{prop}
Let $X\subset \bP^{2n+2}$ be a smooth complete intersection of two quadrics
over $\bR$ with invariant $(2n+1)$. Then $X$ is rational.
\end{prop}
\begin{proof}
The analysis in Section~\ref{subsect:remaining} indicates that complex
conjugation exchanges two classes of quadric $n$-folds associated
to complex conjugate points in $D(\bC)\setminus D(\bR)$. Denote these
by $[Q]$ and $[\bar{Q}]$ -- recall from (\ref{notresidual}) that
$$[Q]\cdot [\bar{Q}]=1.$$
Choosing a suitably general complex representation $Q\subset X_{\bC}$,
the intersection $Q\cap \bar{Q}$ is proper. Then $Q\cap \bar{Q}$ 
consists of a single rational point of $X$ with multiplicity one.
In particular, the hypotheses of Construction II are satisfied.
\end{proof}

\begin{prop}
Let $X\subset \bP^{2n+2}$ be a smooth complete intersection of two quadrics
over $\bR$ with invariant 
$$(2n+2-e,e,1)=(e,1,2n+2-e),  \quad 2\le e \le n+1.$$
Assume that either $e$ is even or $e=n+1$. Then $X$ is rational.
\end{prop}
\begin{proof}
Assume first that $e=n+1$. It follows from (\ref{eisn+1}) in
Section~\ref{subsect:remaining} that $X$ admits {\em three} classes
$$\begin{array}{c|cccc}
   & h^n & Q_1 & Q_2 & Q_3   \\
\hline
h^n  & 4 & 2           &    2  &  2     \\
Q_1 & 2 & 1+(-1)^n &    1  &  1     \\
Q_2 & 2 & 1          &  1+(-1)^n    &  1   \\
Q_3 & 2 & 1          &   1  &   1+(-1)^n  
\end{array}$$
with each $Q_i$ realized by a quadric $n$-fold defined over $\bR$.
Construction III gives rationality in this case.

Now assume that $e$ is even. The formula (\ref{countsigns}) shows that
the numbers of $+1$'s and $-1$'s appearing in the $\iota \in W(D_{2n+3})$
associated with complex conjugation
are as close as possible. If $n$ is
even then we have $n+1$ of the former and $n+2$ of the latter;
when $n$ is odd we have $n+2$ of the former and $n+1$ of the letter. 
Given an $n$-plane $P\subset X_{\bC}$, the formulas in Section~\ref{subsect:planes}
yield
$$w_P \cdot w_{\bar{P}} = (-1)^{n+1}/4$$
whence $P$ and $\bar{P}$ are disjoint in $X_{\bC}$. Thus we may apply
Construction I to conclude rationality.
\end{proof}

\begin{rema}
Remark~\ref{rema:Witt} implies that $X$ does not admit curves 
(or surfaces!) of odd
degree defined over $\bR$.
These would force the existence of lines defined over $\bR$,
which do not exist in this isotopy class.
\end{rema}

\subsection{Remaining six-dimensional case}
\label{subsect:135}

To settle the rationality of six-dimensional
complete intersections of two quadrics $X\subset \bP^6$,
there is one 
remaining case in the Krasnov classification:
$(1,3,5)$.

The sequence of signatures of nonsingular elements of $\{\cQ_t\}$ is:
\begin{align*} 
& (1,8) (2,7) (3,6) (4,5) (5,4) (6,3) (5,4) (6,3) (7,2) \\
& (8,1) (7,2) (6,3) (5,4) (4,5) (3,6) (4,5) (3,6) (2,7).
\end{align*}
The signed permutation is the diagonal matrix
$$\diag(-1,1,-1,{\bf{1}},-1,-1,-1,1,-1)$$
and the invariant cycles are:
$$\begin{array}{c|cccc}
   & h^3 & Q_1 & Q_2 & Q_3  \\
\hline
h^3 & 4 & 2 &    2  &  2     \\
Q_1 & 2 & 0 &    1  &  1    \\
Q_2 & 2 & 1 &    0  &  1   \\
Q_3 & 2 & 1 &    1  &  0   
\end{array}
$$
Here $Q_1$ corresponds to the singular fiber of 
signature $(4,4)$ and $Q_2$ and $Q_3$ correspond to the 
singular fibers of signatures $(2,6)$ and $(6,2)$.  
If $P\subset X$ is a three-plane then $w_P\cdot w_{\bar{P}}=-3/4$
and $P$ and $\bar{P}$ meet in a single point.

\section{Extensions and more general fields}
\label{sect:extensions}
We work over a field $k$ of characteristic zero. In this section, we give further examples 
of rationality constructions for $2n$-dimensional intersections
of two quadrics over $k$, relying on special subvarieties of dimension $n$.

\subsection{Dimension four: intersection computations}
Given $X \subset \bP^6$, a smooth complete intersection of two quadrics, 
we have
\begin{align*}
c_t(\cT_X) & \equiv  (1 + 7ht + 21h^2t^2)/(1+2ht)^2 \pmod{t^3} \\
         & \equiv  (1 + 7ht + 21h^2t^2)(1 - 2ht + 4h^2t^2)^2 \pmod{t^3} \\
         & \equiv  1 + 3h t + 5h^2 t^2 \pmod{t^3} 
\end{align*}
If $T \subset X$ is a smooth projective geometrically connected
surface then 
\begin{align*}
c_t(\cN_{T/X}) & = (1+3ht + 5h^2t^2)/(1 - K_T t + \chi(T) t^2) \\
             & = 1 + (3h+K_T)t + (5h^2 + 3hK_T + K_T^2 - \chi(T))t^2,
\end{align*}
where $\chi(T)$ is the topological Euler characteristic.
The expected dimension of the deformation space of $T$ in $X$ is
\begin{align*}
\chi(\cN_{T/X}) &= 2 - \tfrac{1}{2}(3h+K_T)K_T + \frac{1}{2}(3h+K_T)^2 \\
			& \hskip 5cm -(5h^2 + 3hK_T + K_T^2 - \chi(T)) \\
              &= 2\chi(\cO_T) + \tfrac{1}{2}(-h^2 - 3hK_T) - K_T^2 +  \chi(T)).
\end{align*}
For example,
\begin{itemize}
\item{if $T=\bP^2$ is embedded as a plane then $(T\cdot T)_X=2$ and
$T$ is rigid;}
\item{if $T$ is a quadric then $(T\cdot T)_X = 2$ and $T$ moves
in a three-parameter family;}
\item{if $T$ is a quartic scroll then $(T\cdot T)_X=6$ and
moves in a five-parameter family;}
\item{if $T$ is a sextic del Pezzo surface then $(T\cdot T)_X=12$
and $T$ moves in an eight-parameter family.} 
\end{itemize}

\subsection{Dimension four: surfaces with one apparent double point}
Recall that Construction V gives the rationality of fourfolds
admitting a surface with one apparent double point.
A classical result asserted by Severi -- see \cite[Th.~4.10]{CMR} for a
modern proof -- characterizes smooth surfaces 
$T\subset \bP^5$ with one apparent double
point, i.e., surfaces that acquire one singularity on generic projection into
$\bP^4$:
\begin{itemize}
\item{ $\deg(T)=4$: $T$ is a quartic scroll 
$$ T \simeq \bP(\cO_{\bP^1}(2)^2),  \quad \bP(\cO_{\bP^1}(1) \oplus \cO_{\bP^1}(3));$$
}
\item{ $\deg(T)=5$: $T$ is a quintic del Pezzo surface.}
\end{itemize}
Thus Construction V says that a smooth complete intersection of two quadrics $X\subset \bP^6$
is rational if it contains a quartic scroll, a quintic del Pezzo surface, or a sextic del Pezzo surface
with a rational point. Rationality always holds when there are positive-dimensional
subvarieties of odd degree (see Section~\ref{subsect:genrat}), so
we focus attention to the first case.

We seek criteria for the existence of a quartic scroll $T \subset X$,
defined over $k$. Clearly, the class $[T]$ must be Galois-invariant;
however, Galois-invariant classes need not be represented by 
cycles over $k$.

The intersection computations above imply that
$$
[T]=[Q_2] + [Q_3],
$$
where $Q_2$ and $Q_3$ represent quadric surfaces in $X$, defined
over the algebraic closure.
Assume that the class $[Q_2]+[Q_3]$ is Galois invariant and represents
algebraic cycles defined over the ground field.
We look for quartic scrolls $T\subset X$
with class $[T] = [Q_2]+[Q_3]$.

\begin{rema}
Over $k=\bR$, case $(1,2,4)$ has signed permutation 
$$\diag(1,-1,-1,-1,{\bf{1}},-1,1).$$
The intersection form on
the invariant part of $H^4(X_{\bC},\bZ)$ takes the form
$$\begin{array}{c|cccc}
   & h^2 & Q_1 & Q_2 & Q_3 \\
\hline
h^2 & 4 & 2 &    2  &  2 \\
Q_1 & 2 & 2 &    1  &  1 \\
Q_2 & 2 & 1 &    2  &  1 \\
Q_3 & 2 & 1 &    1  &  2
\end{array}$$
Here we assume $Q_1$ (and $h^2-Q_1$) are associated with the element
of $\{\cQ_t\}$ with signature $(3,3)$ and $Q_2$ and $Q_3$ are
contributed by the elements with signatures $(1,5)$ and $(5,1)$.
While $Q_1$ is definable over $\bR$, $Q_2$ and $Q_3$ are not definable
over $\bR$ (see Section~\ref{subsect:overR}).  

The requisite real cycles exist in $[T]=[Q_2]+[Q_3]$. This follows 
from the exact sequence in \cite[{\S}4]{CT15}, using the rationality
of $X$ over $\bR$ to ensure the vanishing of the unramified cohomology. 
It would be
interesting to deduce this directly using 
cohomological machinery \cite{Kahn96}.
However, we do not know whether such $X$
contain quartic scrolls, in general.
\end{rema}

Let $M$ denote the moduli space of quartic scrolls in a fixed
cohomology class on $X$. There is a morphism
$$
\begin{array}{rcl}
M &\rightarrow &(\bP^6)^{\vee}\\
T & \mapsto & \operatorname{span}(T) 
\end{array}
$$
assigning to each scroll the hyperplane it spans.

Hyperplane sections $Y=H\cap X$ containing such scrolls
are singular by the Lefschetz hyperplane theorem. 
Computations in {\tt Macaulay2} indicate that a generic such $Y$
has four ordinary singularities. If a 
complete intersection of two quadrics $Y\subset \bP^5$
contains a quartic scroll,
it contains two families of such scrolls, each parametrized by
$\bP^3$: These arise from residual intersections in quadrics in
$$I_T(2)/I_Y(2).$$
Thus the residual family has class 
$$2h^2-[T]=(h^2-[Q_2])+(h^2-[Q_3]).$$
The hyperplane sections of $X$ with four singularities 
should be para\-metrized by a reducible surface with distinguished component
$\Sigma \subset (\bP^6)^{\vee}$. 

We speculate that $\Sigma$ is a quartic del Pezzo surface,
constructed as follows: Consider the pencil of quadrics $\cQ_t$ 
defining $X$ and fix the pair of rank-six quadrics
$$\cQ_{t_2},\cQ_{t_3}$$
whose maximal isotropic subspaces sweep out $Q_2$ and $Q_3$.
Let $v_i\in \cQ_{t_i}$ denote the vertices and $\ell$ the line
they span, which is defined over $k$ even when $t_2$ and $t_3$
are conjugate over $k$. Projecting from $\ell$ gives a degree-four
cover
$$X \rightarrow \bP^4.$$
Geometrically, the covering group is the Klein four-group and the 
branch locus consists of two quadric hypersurfaces $Y_2,Y_3$
intersecting in a degree-four del Pezzo surface $S_{23}$.  
Is $\Sigma \simeq S_{23}$ over $k$?

\subsection{Dimension six: Threefolds with one apparent double point}
Construction V indicates that the existence of a threefold $W\subset X$
with one apparent double point yields rationality.
The following classification \cite{CMR} builds on constructions of Edge \cite{Edge}:
\begin{itemize}
\item{$\deg(W)=5$: a scroll in planes associated with two lines and twisted cubic,
or one line and two conics;}
\item{$\deg(W)=6$: an Edge variety constructed as a residual intersection
$$
Q \cap (\bP^1 \times \bP^3) = \Pi_1 \cup \Pi_2 \cup W,
$$
where $Q$ is a quadric hypersurface and
the $\Pi_i\simeq \bP^3$ are fibers of the Segre variety
under the first projection;}
\item{$\deg(W)=7$: an Edge variety constructed as a residual intersection
$$Q \cap (\bP^1 \times \bP^3) = \Pi \cup W$$
with the notation as above;}
\item{$\deg(W)=8$: a scroll in lines over $\bP^2$ of the form
$\bP(\cE)$ (one-dimensional quotients of $\cE$)
where $\cE$ is a rank-two vector bundle given as an extension
$$ 0 \ra \cO_{\bP^2} \ra \cE \ra \cI_{p_1,\ldots,p_8}(4) \ra 0 $$
for eight points $p_1,\ldots,p_8 \in \bP^2$,
no four collinear or seven on a conic.}
\end{itemize}

Given that rationality follows when there are positive-dimensional
subvarieties of odd degree (see Section~\ref{subsect:genrat}) we focus
on the varieties of even degree.

\subsection{Dimension six: Degree six Edge variety}
Let $W\subset \bP^7$ denote an Edge variety arising as follows:
Consider the Segre fourfold
$$
\bP^1 \times \bP^3 \subset \bP^7
$$ 
and take the
residual intersection to two copies of $\bP^3$
$$\{0,\infty \} \times \bP^3 \subset \bP^1 \times \bP^3$$
in a quadric hypersurface.  The resulting threefold 
$$
W\simeq \bP^1 \times \Sigma,
$$
where $\Sigma \subset \bP^3$ is a quadric hypersurface.  
Note that the ideal of $W \subset \bP^7$ is generated by 
nine quadratic forms.  Complete intersections of two quadrics
$$W \subset Y \subset \bP^7$$
depend on five parameters. A {\tt Magma} computation shows that
a generic such $Y$ has eight ordinary singularities. 

Suppose we have an embedding $W \hookrightarrow X$, where $X\subset \bP^8$
is a smooth complete intersection of two quadrics. For fixed $X$, the Hilbert scheme of such threefolds has dimension eight. Realizing 
$$
W\simeq \bP^1 \times \bP^1 \times \bP^1
$$ 
we have
\begin{eqnarray*}
c_1(\cN_{W/X})& =& 3(h_1+h_2+h_3),  \\
c_2(\cN_{W/X})& =& 8(h_1h_2+h_1h_3+h_2h_3),\\
c_3(\cN_{W/X})& =& 4h_1h_2h_3.
\end{eqnarray*}
The Riemann-Roch formula gives $\chi(N_{W/X})=8$.
Since $(W\cdot W)_X =4$, the primitive class 
$$([W] - \tfrac{3}{2} h^3)^2 = 4 - 18+9 = -5.$$

\begin{rema}
Suppose that $X$ is defined over $\bR$, and corresponds to the $(1,2,6)$ case,
with signed permutation matrix (see Section~\ref{subsect:remaining})
$$\diag(-1,1,1,1,-1,{\bf{1}},-1,1,-1).$$
The invariant cycles are:
$$\begin{array}{c|cccccc}
   & h^3 & Q_1 & Q_2 & Q_3 & Q_4 & Q_5  \\
\hline
h^3 & 4 & 2 &    2  &  2   & 2   & 2  \\
Q_1 & 2 & 0 &    1  &  1   & 1   & 1  \\
Q_2 & 2 & 1 &    0  &  1   & 1   & 1 \\
Q_3 & 2 & 1 &    1  &  0   & 1   & 1 \\
Q_4 & 2 & 1 &    1  &  1   & 0   & 1 \\
Q_5 & 2 & 1 &    1  &  1   & 1   & 0
\end{array}
$$
Here $Q_1$ corresponds to the singular fiber of
signature $(4,4)$ and the classes  $Q_2,\ldots,Q_5$ correspond to the
singular fibers of signatures $(2,6)$ and $(6,2)$.

The class
$$[Q_2]+[Q_3]+[Q_4]+[Q_5]-[Q_1]$$
has degree six and self-intersection four. Is it represented by cycles defined over
$\bR$? Does it admit an Edge variety of degree six over $\bR$?
Over more general $k$ where the requisite cycles exist?
\end{rema}

\subsection{Dimension six: Degree eight variety}
Let $\cE$ be a stable rank-two vector bundle on $\bP^2$ with invariants
$c_1(\cE)=4L$ and $c_2(\cE)=8L^2$. Note that $\Gamma(\cE)$ has dimension
eight, giving an inclusion
$$V:=\bP(\cE^{\vee}) \subset \bP^7.$$
We have a tautological exact sequence
$$
0 \ra \cO_V(-\xi) \ra \cE^{\vee}_V \ra Q \ra 0,
$$
whence
$$0 \ra Q(\xi)  \ra T_V \ra T_{\bP^2} \ra 0.$$
Thus we have the following
\begin{align*} 
\xi^2 - 4L\xi + 8L^2 &=0 \\
c(Q(\xi))&= 1 + (2\xi-4L) + (\xi^2 - 4L\xi+ 8L^2), \\
c(\cT_V)&= (1+3L+3L^2)(c(Q(\xi)) \\
      &= 1 + (2\xi-L) + (\xi^2-4L\xi+8L^2 + 6\xi L-12L^2 + 3L^2)
\end{align*}
and we find
\begin{eqnarray*}
c_1(\cN_{V/X}) & = & 3\xi +L,\\
c_2(\cN_{V/X})  & = & 5\xi^2 - \xi L +2L^2, \\
c_3(\cN_{V/X})   & = & 3\xi^3 -3\xi^2 L + 2L^2\xi -9L^3.
\end{eqnarray*}
Note that 
$$
\deg(L^3)=0, \quad \deg(L^2\xi)=1,\quad  \deg(L\xi^2)=4, \text{ and } 
\deg(\xi^3)=8,
$$ 
so we conclude that
$$(V\cdot V)_X =14.$$
The primitive class $[V]-2h^3$ has self-intersection $14-4\cdot 8+16=-2$
which means that
$$
[V]=h^3 + [Q_2] + [Q_3], \quad (Q_2\cdot Q_3)=1,
$$
where $Q_2$ and $Q_3$ are classes of quadric threefolds in $X$,
defined over the algebraic closure. (Up to the action of the Weyl group $W(D_9)$ this
is the only possibility.)

Returning to the only remaining case in dimension six where
rationality over $\bR$ remains open
(see Section~\ref{subsect:135}):
\begin{ques}
Let $X \subset \bP^8$ be a smooth complete intersection of two quadrics
over $\bR$ in isotopy class $(1,3,5)$.
Which classes of codimension-three cycles $X$ are realized
over $\bR$?
Are there varieties with one apparent double point, defined over $\bR$, representing these
classes? 
\end{ques}

\bibliographystyle{alpha}
\bibliography{cases}

\begin{thebibliography}{CTSSD87}

\bibitem[BD85]{BD}
Arnaud Beauville and Ron Donagi.
\newblock La vari\'{e}t\'{e} des droites d'une hypersurface cubique de
  dimension {$4$}.
\newblock {\em C. R. Acad. Sci. Paris S\'{e}r. I Math.}, 301(14):703--706,
  1985.

\bibitem[BW19]{BenWit2}
Olivier Benoist and Olivier Wittenberg.
\newblock Intermediate {J}acobians and rationality over arbitrary fields, 2019.
\newblock \texttt{arXiv:1909.12668}.

\bibitem[BW20]{BenWit1}
Olivier Benoist and Olivier Wittenberg.
\newblock The {C}lemens-{G}riffiths method over non-closed fields.
\newblock {\em Algebr. Geom.}, 7(6):696--721, 2020.

\bibitem[CMR04]{CMR}
Ciro Ciliberto, Massimiliano Mella, and Francesco Russo.
\newblock Varieties with one apparent double point.
\newblock {\em J. Algebraic Geom.}, 13(3):475--512, 2004.

\bibitem[CT96]{CT}
Jean-Louis Colliot-Th\'{e}l\`ene.
\newblock Groupes lin\'{e}aires sur les corps de fonctions de courbes
  r\'{e}elles.
\newblock {\em J. Reine Angew. Math.}, 474:139--167, 1996.

\bibitem[CT15]{CT15}
Jean-Louis Colliot-Th\'{e}l\`ene.
\newblock Descente galoisienne sur le second groupe de {C}how: mise au point et
  applications.
\newblock {\em Doc. Math.}, (Extra vol.: Alexander S. Merkurjev's sixtieth
  birthday):195--220, 2015.

\bibitem[CTSSD87]{CTSSDI}
Jean-Louis Colliot-Th\'{e}l\`ene, Jean-Jacques Sansuc, and Peter
  Swinnerton-Dyer.
\newblock Intersections of two quadrics and {C}h\^{a}telet surfaces. {I}.
\newblock {\em J. Reine Angew. Math.}, 373:37--107, 1987.

\bibitem[Duc98]{Ducros}
Antoine Ducros.
\newblock L'obstruction de r\'{e}ciprocit\'{e} \`a l'existence de points
  rationnels pour certaines vari\'{e}t\'{e}s sur le corps des fonctions d'une
  courbe r\'{e}elle.
\newblock {\em J. Reine Angew. Math.}, 504:73--114, 1998.

\bibitem[Edg32]{Edge}
W.~L. Edge.
\newblock The number of apparent double points of certain loci.
\newblock {\em Proc. Cambridge Philos. Soc.}, 28:285--299, 1932.

\bibitem[HT19a]{HT1}
Brendan Hassett and Yuri Tschinkel.
\newblock Cycle class maps and birational invariants, 2019.
\newblock \texttt{arXiv:1908.00406} to appear in {\em Communications on Pure
  and Applied Mathematics}.

\bibitem[HT19b]{HTplusCT}
Brendan Hassett and Yuri Tschinkel.
\newblock Rationality of complete intersections of two quadrics, 2019.
\newblock \texttt{arXiv:1903.08979} with an appendix by J.L.
  Colliot-Th\'el\`ene, to appear in {\em L'Enseignement Math\'ematique}.

\bibitem[HT19c]{HT2}
Brendan Hassett and Yuri Tschinkel.
\newblock Rationality of {F}ano threefolds of degree 18 over nonclosed fields,
  2019.
\newblock \texttt{arXiv:1910.13816} to appear in the Schiermonnikoog
  rationality volume.

\bibitem[Isk79]{Isk}
V.~A. Iskovskih.
\newblock Minimal models of rational surfaces over arbitrary fields.
\newblock {\em Izv. Akad. Nauk SSSR Ser. Mat.}, 43(1):19--43, 237, 1979.

\bibitem[Kah96]{Kahn96}
Bruno Kahn.
\newblock Applications of weight-two motivic cohomology.
\newblock {\em Doc. Math.}, 1:No. 17, 395--416, 1996.

\bibitem[Kot00]{Kottwitz}
Robert~E. Kottwitz.
\newblock Involutions in {W}eyl groups.
\newblock {\em Represent. Theory}, 4:1--15, 2000.

\bibitem[KP20a]{KP1}
Alexander Kuznetsov and Yuri Prokhorov.
\newblock Rationality of {F}ano threefolds over non-closed fields, 2020.
\newblock \texttt{arXiv:1911.08949}.

\bibitem[KP20b]{KP2}
Alexander Kuznetsov and Yuri Prokhorov.
\newblock Rationality of {M}ukai varieties over non-closed fields, 2020.
\newblock \texttt{arXiv:2003.10761}.

\bibitem[Kra18]{Krasnov18}
V.~A. Krasnov.
\newblock On the intersection of two real quadrics.
\newblock {\em Izv. Ross. Akad. Nauk Ser. Mat.}, 82(1):97--150, 2018.
\newblock translation in {I}zv. {M}ath. 82 (2018), no. 1, 91–-139.

\bibitem[Lee07]{Leep}
David~B. Leep.
\newblock The {A}mer-{B}rumer theorem over arbitrary fields, 2007.
\newblock available at
  \url{http://www.ms.uky.edu/~leep/Amer-Brumer_theorem.pdf}.

\bibitem[Rei72]{ReidThesis}
Miles Reid.
\newblock {\em The complete intersection of two or more quadrics}.
\newblock PhD thesis, Trinity College, Cambridge, 1972.
\newblock available at
  \url{http://homepages.warwick.ac.uk/~masda/3folds/qu.pdf}.

\bibitem[Rus00]{Russo}
Francesco Russo.
\newblock On a theorem of {S}everi.
\newblock {\em Math. Ann.}, 316(1):1--17, 2000.

\bibitem[Sch96]{Sch}
Claus Scheiderer.
\newblock Hasse principles and approximation theorems for homogeneous spaces
  over fields of virtual cohomological dimension one.
\newblock {\em Invent. Math.}, 125(2):307--365, 1996.

\end{thebibliography}

\end{document}